\documentclass[11pt]{amsart}
\usepackage{amscd}      
\usepackage{amssymb}
\usepackage{amsmath, amsthm, graphics}
\usepackage{pb-diagram}
\allowdisplaybreaks[3]

\usepackage[margin=1.1in]{geometry}

\usepackage{latexsym}
\usepackage{epsfig}
\usepackage{amsfonts}
\usepackage{enumerate}
\usepackage{times}

\newtheorem{thm}{Theorem}[section]

\newtheorem{lem}[thm]{Lemma}
\newtheorem{cor}[thm]{Corollary}

\newtheorem{rmk}[thm]{Remark}
\newtheorem{ex}[thm]{Example}

\newtheorem{conj}[thm]{Conjecture}

\newcommand{\Y}{\mathcal{Y}}

\newcommand{\B}{\mathcal{B}}

\def\<{\left\langle}
\def\>{\right\rangle}

\newcommand{\call}{{\mathcal P}ic}
\newcommand{\calm}{{\mathcal M}}

\newcommand{\calo}{{\mathcal O}}
\newcommand{\calp}{{\mathcal P}}

\newcommand{\cals}{{\mathcal S}}
\newcommand{\calt}{{\mathcal T}}
\newcommand{\calv}{{\mathcal V}}

\newcommand{\calw}{{\mathcal W}}

\newcommand{\complex}{{\mathbb C}}

\newcommand{\reals}{{\mathbb R}}

\newcommand{\integers}{{\mathbb Z}}

\newcommand{\Aut}{\mbox{$Aut$}}
\newcommand{\ad}{\mbox{$Ad$}}
\newcommand{\aut}{\mbox{$\mathcal{A}ut$}}
\newcommand{\bit}{\rm Btr}
\newcommand{\tor}{\rm Tor}

\title{On the relative dual of an $S^1$-gerbe over an orbifold}
\author{Ilya Shapiro}
\address{Department of Mathematics and Statistics\\ University of Windsor\\ 401 Sunset Avenue\\Windsor, Ontario N9B 3P4\\ Canada}
\email{ishapiro@uwindsor.ca}

\author{Xiang Tang}
\address{Department of Mathematics\\ Washington University\\ St. Louis\\ MO 63130\\ USA}
\email{xtang@math.wustl.edu}

\author{Hsian-Hua Tseng}
\address{Department of Mathematics\\ Ohio State University\\ 100 Math Tower, 231 West 18th Ave.\\Columbus\\ OH 43210\\ USA}
\email{hhtseng@math.ohio-state.edu}
\begin{document}

\dedicatory{Dedicated to Professor Marc Rieffel on the occasion of his $75$th birthday}

\date{}
\maketitle

\begin{abstract}
We construct a new effective orbifold $\widehat{\Y}$ with an $S^1$-gerbe $c$ to study an $S^1$-gerbe $\mathfrak{t}$ on a $G$-gerbe $\Y$ over an orbifold $\B$. We view the former as the relative dual, relative to $\B$, of the latter.  We show that the two pairs $(\Y, \mathfrak{t})$ and $(\widehat{\Y}, c)$ have isomorphic categories of sheaves, and also the associated twisted groupoid algebras are Morita equivalent. As a corollary,  the K-theory and cohomology groups of $(\Y, \mathfrak{t})$ and $(\widehat{\Y}, c)$ are isomorphic.
\end{abstract}

\section{Introduction}
Gerbes are useful tools for studying
various problems in mathematics and physics. They can
be applied to the theory of non-abelian cohomology \cite{gi}, loop
spaces and characteristic classes \cite{bry}, the Dixmier-Douady
class, continuous trace $C^*$-algebras, index theory (\cite{ra-wi}
and \cite{ca-mi-mu}), the comparison between the Brauer group and
the cohomological Brauer group of a scheme \cite{de-Jong}, and the
period-index problem \cite{Lieblich2}. Furthermore, in physics,
gerbes are intimately connected with the study of discrete torsion.
See, e.g., \cite{vw} and \cite{sharpe}.

In this article, we study an $S^1$-gerbe over an orbifold $\Y$. An orbifold $\Y$ can be presented \cite{mo-pr} by a proper \'etale groupoid $\mathfrak{H}$ with $\Y$ being the quotient of $\mathfrak{H}_0$ by the equivalence relation defined by $\mathfrak{H}$. And an $S^1$-gerbe over $\Y$ can be represented \cite{be-xu} by an $S^1$ central extension of the groupoid $\mathfrak{H}$.  Generalizing the correspondence between $S^1$ central extensions of a group $H$ and the cohomology group $H^2(H, S^1)$, we can present \cite{be-xu} an $S^1$-gerbe over $\Y$ by an $S^1$-valued 2-cocycle $\mathfrak{t}$ on the groupoid $\mathfrak{H}$ presenting $\Y$, by passing to a Morita equivalent presentation. Inspired by the recent developments of duality of gerbes (\cite{hel-hen-pan-sh} and \cite{tt}), our goal in this paper is to construct a new effective orbifold $\widehat{\Y}$ together with an $S^1$-gerbe $c$ over $\widehat{\Y}$ out of $(\Y, \mathfrak{t})$ such that the following conjecture holds.

\begin{conj}The geometry/topology of $(\Y, \mathfrak{t})$ is equivalent to the geometry/topology of $(\widehat{\Y}, c)$.
\end{conj}

Our construction of $(\widehat{\Y}, c)$ is inspired by \cite[Prop. 4.6]{bn}. Namely, there is a finite group $G$ and a reduced orbifold $\B$ such that $\Y$ is a $G$-gerbe over $\B$.  In \cite{tt}, inspired by \cite{hel-hen-pan-sh}, we constructed an orbifold $\widetilde{\Y}$ with a flat $S^1$-gerbe $\tilde{c}$. And we showed that many important geometric/topological properties of $\Y$ are equivalent to the ones of $(\widetilde{\Y}, \tilde{c})$. As $\mathfrak{t}$ defines an $S^1$-gerbe over $\Y$, our initial thought following the philosophy developed in \cite{tt} was that the dual associated to $\mathfrak{t}$ over $\Y$ should be some data over $(\widetilde{\Y}, \tilde{c})$. However, this idea does not work so well with the following example. Let $\Y$ be the orbifold that is the quotient of a point by a finite group $G$. An $S^1$-gerbe $\Y$ can be presented by an $S^1$-valued 2-cocycle $\mathfrak{t}$ over $G$.  The dual of $\Y$ constructed in \cite{tt} is $\widehat{G}$, the finite set of isomorphism classes of irreducible unitary $G$-representations with a trivial $\tilde{c}$. For the $S^1$-gerbe defined by $\mathfrak{t}$ over $\Y$, its category of sheaves is the category of $\mathfrak{t}$-twisted representations of $G$. However, we do not see a proper connection between this category and the space $\widehat{G}$.

Instead of working with $(\widetilde{\Y}, \tilde{c})$, we apply the key ideas in \cite{tt}.  In the case of the above example $(\Y=[pt/G], \mathfrak{t})$, we consider the set $\widehat{\Y}=\widehat{G}^\mathfrak{t}$ of isomorphism classes of irreducible $\mathfrak{t}$-twisted unitary $G$-representations. We equip $\widehat{G}^\mathfrak{t}$ with the trivial $S^1$-gerbe. The category of $\mathfrak{t}$-twisted $G$-representations is semisimple and therefore is isomorphic to the category of  sheaves over $\widehat{G}^\mathfrak{t}$. By passing to Morita equivalent ones, we present $\B$ by a proper \'etale groupoid $\mathfrak{Q}$ over $M$, and $\Y$ by a groupoid $\mathfrak{H}$ with the extension
\begin{equation}\label{eq:g-gerbe}
G\times M\longrightarrow \mathfrak{H}\longrightarrow \mathfrak{Q}.
\end{equation}
Associated to $(\Y,\mathfrak{t})$, we have a groupoid extension
\begin{equation}\label{eq:g-twisted}
G\times_\mathfrak{t}M\longrightarrow \mathfrak{H}_\mathfrak{t}\longrightarrow \mathfrak{Q},
\end{equation}
where $\mathfrak{H}_\mathfrak{t}$ is the central extension of $\mathfrak{H}$ over $M$ that presents the orbifold $\Y$, and $G\times_\mathfrak{t}M$ is the central extension of $G\times M$ by restriction of $\mathfrak{t}$ to $G\times M$. An important new property in (\ref{eq:g-twisted}) as compared to (\ref{eq:g-gerbe})  is that $\mathfrak{t}$ on $G\times M$ is not constant, and therefore $G\times_\mathfrak{t}M$ is not a bundle of the same group. Therefore, extension (\ref{eq:g-twisted}) is not a gerbe over $\B$ with isomorphic fiber groups any more. We observe a key property that the restriction of $\mathfrak{t}$ to $G\times M$ is locally constant. And it is this property that allows us to define a new orbifold $\widehat{\Y}$  together with a smooth $S^1$-gerbe $c$ generalizing the construction that was developed in \cite{tt}.

Our main result shows that $(\Y, \mathfrak{t})$ and $(\widehat{\Y}, c)$ are isomorphic as noncommutative geometric spaces. We verify this by proving the following two results.
\begin{enumerate}
\item (Theorem \ref{thm:sheaves-eq}) The category of (coherent) sheaves over $(\Y, \mathfrak{t})$ is isomorphic to the category of (coherent) sheaves over $(\widehat{\Y}, c)$.
\item (Theorem \ref{thm:gpd-morita}) The (twisted) groupoid algebra associated to $(\Y, \mathfrak{t})$ is Morita equivalent to the (twisted) groupoid algebra associated to $(\widehat{\Y}, c)$.
\end{enumerate}
As a corollary, the $K$-theory and cohomologies of $(\Y, \mathfrak{t})$ and $(\widehat{\Y}, c)$ are isomorphic. We point out that the new orbifold $\widehat{\Y}$ is effective. Therefore, our results suggest that it is sufficient to study the effective orbifold $\widehat{\Y}$ with the $S^1$-gerbe  $c$ to understand the geometry/topology of a general orbifold $\Y$ with an $S^1$-gerbe $\mathfrak{t}$.

 In general, we expect that all the results we have developed in \cite{tt} naturally extend to the pair $(\Y, \mathfrak{t})$ and $(\widehat{\Y}, c)$. However, the new property from \cite{tt} is that $c$ is in general not flat. Hence, the result we used in \cite{tt} about deformation quantizations of groupoid algebras \cite{ta} can not be applied directly here.  We plan to come back to study this issue using the ideas in \cite{bgnt}.

This article is organized as follows. In Sec. \ref{sec:group}, we explain our construction of the dual $(\widehat{\Y}, c)$ and the main theorems in the special case that $\Y$ is a $G$-gerbe over the orbifold $BQ$, where $Q$ is a finite group. And in Sec. \ref{sec:orbifold}, we develop the general construction for $(\widehat{\Y}, c)$ and prove the main results explained above. We work in the full generality of $S^1$-gerbes on a $G$-gerbe over an orbifold, namely, the groupoid extension (\ref{eq:g-twisted}), which covers the $S^1$-gerbe over an orbifold as a special example.  In the Appendix, we discuss the results of this article in the language of stacks. \\

\noindent{\bf Acknowledgements}: Shapiro's research is partially supported by an NSERC grant. Tang's research is partially supported by the NSF grant DMS 0900985 and NSA grant H98230-13-1-0209. Tseng's research is partially supported by a Simons Foundation collaboration grant.

\section{Central extension of a group extension}\label{sec:group}

\subsection{Structure of extensions}\label{subsec:structure}
Let $G$ and $Q $ be finite groups. Denote by $\Aut(G)$  the automorphism group of $G$, and let $\ad_g$  be the inner automorphism of $G$ given by the group element $g\in G$. Consider a map $\rho: Q\to \Aut(G)$. We assume that $\rho$ is almost a group morphism. More precisely, there is a map $\tau:Q\times Q\to G$ such that
\[
\rho(q_1)\rho(q_2)=\ad_{\tau(q_1,q_2)}\rho(q_1q_2),\qquad \tau(q_1, q_2)\tau(q_1q_2, q_3)=\rho_{q_1}(\tau(q_2, q_3))\tau(q_1, q_2q_3).
\]
As is explained in \cite[Section 3]{tt}, the above data $(\rho, \tau)$ determine an extension of $Q$ by $G$, i.e.,
\[
1\rightarrow G\rightarrow H\rightarrow Q\rightarrow 1,
\]
where $H$ is isomorphic to $G\times Q$ with the following product,
\[
(g_1, q_1)(g_2, q_2)=(g_1\rho(q_1)(g_2)\tau(q_1, q_2), q_1q_2).
\]

In this section, we consider a central extension of $H$ by the circle group $S^1$,
\[
1\rightarrow S^1\rightarrow H_\mathfrak{t}\rightarrow H\rightarrow 1.
\]
Such an extension group $H_\mathfrak{t}$ is determined by an $S^1$-valued 2-cocycle $\mathfrak{t}$ on $H$.  More precisely, $H_\mathfrak{t}$ is isomorphic to $H\times S^1$ whose multiplication is defined by
\[
(h_1, s_1)(h_2, s_2)=(h_1h_2, \mathfrak{t}(h_1, h_2)s_1s_2),
\]
and $\mathfrak{t}$ satisfies
\[
\mathfrak{t}(h_1, h_2)\mathfrak{t}(h_1h_2, h_3)=\mathfrak{t}(h_1, h_2h_3)\mathfrak{t}(h_2, h_3).
\]

The $S^1$-valued 2-cocycle $\mathfrak{t}$ restricts to define a central extension $G_\mathfrak{t}$ of $G$, i.e.,
\[
1\rightarrow S^1\rightarrow G_\mathfrak{t} \rightarrow G\rightarrow 1.
\]
The group $G_\mathfrak{t}$ is a normal subgroup of $H_\mathfrak{t}$ with $Q$ being the quotient, i.e.,
\[
1\rightarrow G_\mathfrak{t} \rightarrow H_\mathfrak{t}\rightarrow Q\rightarrow 1.
\]

For simplicity, we will assume that both $\tau$ and $\mathfrak{t}$ are normalized, i.e.,
$$
\tau(1, q)=\tau(q,1)=\tau(q,q^{-1})=1,$$  $$\mathfrak{t}\big((1,1), (g,q)\big)=\mathfrak{t}\big((g,q), (1,1)\big)=\mathfrak{t}\big((g,q), (\rho^{-1}_q(g^{-1}),q^{-1})\big)=1.$$

Recall that $H_\mathfrak{t}$ can be written as $S^1\times_\mathfrak{t} H$.  Furthermore, $H$ can be written as $G\times_{\rho, \tau} Q$, and so $H_\mathfrak{t}$ can be written as $S^1\times_\mathfrak{t} (G\times_{\rho, \tau} Q)$.  Choose a section $s:Q\longrightarrow S^1\times_\mathfrak{t} (G\times_{\rho, \tau} Q)$ by setting
\[
s(q)=(1,1,q).
\]
We compute that
\[
s(q_1)s(q_2)=(1,1,q_1)(1,1,q_2)=\Big(\mathfrak{t}\big((1, q_1), (1,q_2)\big), \tau(q_1, q_2), q_1q_2\Big).
\]
Therefore, $s(q_1)s(q_2)s(q_1q_2)^{-1}$ is equal to
\[
\Big(\mathfrak{t}\big((1,q_1), (1, q_2)\big)\mathfrak{t}\big((\tau(q_1,q_2), q_1q_2), (1, (q_1q_2)^{-1})\big), \tau(q_1,q_2),1\Big).
\]
Define $\sigma: Q\times Q\rightarrow S^1$ by
\[
\sigma(q_1, q_2):=\mathfrak{t}\big((1,q_1), (1, q_2)\big)\mathfrak{t}\big((\tau(q_1,q_2), q_1q_2), (1, (q_1q_2)^{-1})\big).
\]
And we have
\[
s(q_1)s(q_2)=\big(\sigma(q_1,q_2), \tau(q_1, q_2), 1\big) s(q_1q_2).
\]
Let
\[
\bar{\tau}(q_1, q_2)=\big(\sigma(q_1, q_2), \tau(q_1, q_2)\big)\in G_\mathfrak{t}.
\]

For $q\in Q$, and $(s,g)\in G_\mathfrak{t}$, we compute $\ad_{q}(s,g)$ to be
\begin{align*}
(1,1,q)(s,g,1)(1,1,q^{-1})&=\Big(s \mathfrak{t}\big((1,q), (g,1)\big), \rho_q(g), q\Big)(1,1,q^{-1})\\
&=\Big( s \mathfrak{t}\big((1,q), (g,1)\big)\mathfrak{t}\big((\rho_q(g),q), (1,q^{-1})\big), \rho_q(g),1\Big).
\end{align*}
Define $\nu: Q\times G\to S^1$ by
\[
\nu(q,g)=\mathfrak{t}\big((1,q), (g,1)\big)\mathfrak{t}\big((\rho_q(g),q), (1,q^{-1})\big).
\]
Then we have
\[
\ad_{q}(s,g)=(s\nu(q,g), \rho_q(g)).
\]
As $\ad_{q_1}\circ \ad_{q_2}=\ad_{(\sigma(q_1,q_2), \tau(q_1, q_2))} \ad_{q_1q_2}$, we have
\begin{align*}
\nu(q_1, \rho_{q_2}(g))\nu(q_2, g)&=\mathfrak{t}\big((\tau(q_1, q_2),1), (\rho_{q_1q_2}(g), 1)\big)\\
&=\mathfrak{t}\big((\tau(q_1,q_2)\rho_{q_1q_2}(g), 1),(\tau(q_1, q_2)^{-1},1) \big)\nu(q_1q_2, g).
\end{align*}

It is straightforward  to check the following properties for $\bar{\tau}:Q\times Q\to G_\mathfrak{t}$:
\begin{align*}
&(1)\qquad\bar{\tau}(1, q)=\bar{\tau}(q,1)=\bar{\tau}(q, q^{-1})=1\in G_\mathfrak{t},\\
&(2)\qquad s(q_1)s(q_2)=(\bar{\tau}(q_1, q_2), 1)s(q_1q_2)\in H_\mathfrak{t},\ \text{and therefore}\ \ad_{q_1}\ad_{q_2}=\ad_{\bar{\tau}(q_1, q_2)}\ad_{q_1q_2},\\
&(3)\qquad\bar{\tau}(q_1, q_2)\bar{\tau}(q_1q_2, q_3)=\ad_{q_1}\big(\bar{\tau}(q_2, q_3)\big)\bar{\tau}(q_1, q_2q_3)\in G_\mathfrak{t}.
\end{align*}


\subsection{Group algebra and twisted representations}\label{subsec:group-algebra}

With the discussion of the group structure of $H_\mathfrak{t}$, we now describe the twisted group algebra $\complex (H, \mathfrak{t})$.  Recall that the twisted group algebra $\complex(H,\mathfrak{t})$ is spanned by group elements of $H$ with the following multiplication
\[
\sum_{i}a_ih_i\sum_j b_jh_j =\sum_{i,j}\mathfrak{t}(h_i, h_j) a_ib_j h_ih_j,
\]
where $S^1$ is naturally embedded in $\mathbb{C}$ as the unit circle.

A $\mathfrak{t}$-twisted representation of $G$ refers to a representation of the group $G_\mathfrak{t}$ such that the subgroup $S^1$ acts with weight $1$, i.e. $\rho: G\to GL(V)$ such that
\[
\rho(g_1)\rho(g_2)=\rho(g_1g_2)\mathfrak{t}(g_1, g_2).
\]
Let $\widehat{G}^\mathfrak{t}$ be the set of isomorphism classes of irreducible unitary $\mathfrak{t}$-twisted representations of $G$. For $[\rho]\in \widehat{G}^\mathfrak{t}$, fix a representative $\rho: G_\mathfrak{t} \to \mathfrak{U}(V_\rho)$, where $\mathfrak{U}(V_\rho)$ is the unitary group of the Hilbert space $V_\rho$.  As $\complex (G, \mathfrak{t})$ is semisimple, we have the isomorphism
\begin{equation}
\label{eq:twisted-isom}
\complex (G, \mathfrak{t})\cong \bigoplus_{[\rho]\in \widehat{G}^\mathfrak{t}} \operatorname{End}(V_\rho).
\end{equation}
Fix a $q\in Q$, for $\rho: G_\mathfrak{t} \to \mathfrak{U}(V_\rho)$, the map $\rho\circ \ad_q: G_\mathfrak{t} \to \mathfrak{U}(V_\rho)$ defines a representation of $G_\mathfrak{t}$.  As $Q$ acts on $G$ modulo inner automorphisms, $[\rho]\mapsto [\rho\circ \ad_q]$ defines a natural $Q$-action on $\widehat{G}^\mathfrak{t}$.

Since the representation $\rho\circ \ad_q: G_\mathfrak{t}\to \mathfrak{U}(V_\rho)$ is isomorphic to the pre-chosen representation $q([\rho]): G_\mathfrak{t}\to \mathfrak{U}(V_{q([\rho])})$, there is an isomorphism of vector spaces
\[
T^{[\rho]}_q: V_\rho \to V_{q([\rho])},
\]
such that $\rho(\ad_q(t, g))={T^{[\rho]}_q}^{-1}\circ q([\rho])(t,g)\circ T^{[\rho]}_q$.  Similarly to the developments in \cite[Section 3]{tt}, there exists an $S^1$-valued 2-cocycle  $c$ on the transformation groupoid $\widehat{G}^\mathfrak{t}\rtimes Q\rightrightarrows \widehat{G}^\mathfrak{t}$ such that
\begin{equation}\label{eq:c-dfn}
T^{q_1([\rho])}_{q_2}\circ T^{[\rho]}_{q_1}=c^{[\rho]}(q_1, q_2)T^{[\rho]}_{q_1q_2}\rho(\tau(q_1, q_2))^{-1}\sigma(q_1, q_2)^{-1}.
\end{equation}

\begin{thm}\label{thm:gp-morita}
The twisted group algebra $\complex(H, \mathfrak{t})$ is Morita equivalent to the twisted groupoid algebra $C(\widehat{G}^\mathfrak{t}\rtimes Q, c)$.
\end{thm}
\begin{proof}
The proof is a straightforward generalization of the one of \cite[Theorem 3.1]{tt} and consists of two steps.
\begin{enumerate}
\item With the help of Equation (\ref{eq:twisted-isom}), write $\complex(H, \mathfrak{t})$ as a crossed product
\[
\complex(H, \mathfrak{t})\cong \bigoplus_{[\rho]\in \widehat{G}^\mathfrak{t}} \operatorname{End}(V_\rho)\rtimes_{T, \bar{\tau}} Q.
\]
\item Notice  that $\operatorname{End}(V_\rho)$ is Morita equivalent to $\complex$ with $V_\rho$ being the Morita equivalent bimodule.
Generalizing this to $ \bigoplus_{[\rho]\in \widehat{G}^\mathfrak{t}} \operatorname{End}(V_\rho)\rtimes_{T, \bar{\tau}} Q$, we show that $\oplus_{[\rho]\in \widehat{G}^\mathfrak{t}}V_\rho\otimes \complex Q$  is a Morita equivalence bimodule between $\bigoplus_{[\rho]\in \widehat{G}^\mathfrak{t}} \operatorname{End}(V_\rho)\rtimes_{T, \bar{\tau}} Q$ and $C(\widehat{G}^\mathfrak{t}\rtimes Q, c)$.
\end{enumerate}
We leave the details to the interested reader.
\end{proof}

We decompose $\widehat{G}^\mathfrak{t}$ into a disjoint union of $Q$-orbits $\coprod \calo_i$. The groupoid $\widehat{G}^\mathfrak{t}$ and the cocycle $c$ decomposes accordingly, i.e.,
\[
(\widehat{G}^\mathfrak{t}\rtimes Q, c)=\coprod (\calo_i\rtimes Q, c_i).
\]
For each $i$, choose $[\rho_i]\in \calo_i$, and let $Q_i$ be the isotropy group of the $Q$ action at $[\rho_i]$. Let $\mathfrak{c}_i$ be the restriction of $c_i$ to $Q_i$. Then \cite[Theorem 3.4]{tt} shows that the twisted groupoid algebra $C(\calo_i \rtimes Q, c_i)$ is Morita equivalent to the twisted group algebra $\complex(Q_i, \mathfrak{c}_i)$. And we conclude that
the twisted group algebra $\complex(H, \mathfrak{t})$ is Morita equivalent to
\[
\bigoplus_i \complex(Q_i, \mathfrak{c}_i).
\]
In summary, the above result suggests that the geometry/topology of the $S^1$-gerbe on $BH=[pt/H]$ defined by $\mathfrak{t}$ is isomorphic to the geometry/topology of the disjoint union of the $S^1$-gerbes on $BQ_i=[pt/Q_i]$ defined by $\mathfrak{c}_i$.

\subsection{Category of representations}\label{sec:gp-rep}
In this subsection, we discuss another proof of Theorem \ref{thm:gp-morita} emphasizing categories of representations. The category $\mathfrak{Rep}\big(\complex(H, \mathfrak{t})\big)$ of representations of $\complex(H, \mathfrak{t})$ is isomorphic to the category of $\mathfrak{t}$-twisted representations of $H$, which is the same as the category $\mathfrak{Rep}_\mathfrak{t}(H)$ of representations of $H_\mathfrak{t}$ that the component $S^1$ acts with weight 1. And the category $\mathfrak{Rep}\big(C(\widehat{G}^\mathfrak{t}\rtimes, Q, c)\big)$ of representations of $C(\widehat{G}^\mathfrak{t}\rtimes Q, c)$ is isomorphic to the category $\mathfrak{Rep}(\widehat{G}^\mathfrak{t}\rtimes Q, c)$ of $c$-twisted sheaves over the groupoid $\widehat{G}^\mathfrak{t}\rtimes Q$. Below, we will directly construct an isomorphism between $\mathfrak{Rep}_\mathfrak{t}(H)$ and $\mathfrak{Rep}(\widehat{G}^\mathfrak{t}\rtimes Q, c)$ and therefore give a proof of Theorem \ref{thm:gp-morita}.

We summarize the description of $H_\mathfrak{t}$ in Sec. \ref{subsec:structure}.  $H_\mathfrak{t}$ is an extension of $Q$ by $G_\mathfrak{t}$ with the following data.
\begin{enumerate}
\item A map $\ad: Q\to Aut(G_\mathfrak{t})$.
\item A map $\bar{\tau}: Q\times Q\to G_\mathfrak{t}$ satisfying
\[
\bar{\tau}(q_1, q_2)=\big(\sigma(q_1, q_2), \tau(q_1, q_2)\big).
\]
\item $\bar{\tau}$ and $\ad$ satisfy the following relations
\[
\ad_{q_1}\ad_{q_2}=\ad_{\bar{\tau}(q_1,q_2)}\ad_{q_1q_2},\qquad \bar{\tau}(q_1, q_2)\bar{\tau}(q_1q_2, q_3)=\ad_{q_1}(\bar{\tau}(q_2, q_3))\bar{\tau}(q_1, q_2q_3).
\]
\end{enumerate}

For every $[\rho]\in \widehat{G}^\mathfrak{t}$, fix a representative $\rho: G_\mathfrak{t}\to \mathfrak{U}(V_\rho)$ such that $S^1$ acts with weight $1$, where $\mathfrak{U}(V_\rho)$ is the group of unitary operators on $V_\rho$.  Let $(\pi, V)$ be a representation of $H_\mathfrak{t}$ such that $S^1$ acts with weight $1$.  Consider the natural decomposition
\[
V=\bigoplus_{[\rho]} \operatorname{Hom}^{G_\mathfrak{t}}(V_\rho, V)\otimes V_\rho,
\]
where $\operatorname{Hom}^{G_\mathfrak{t}}(V_\rho, V)$ is the space of $G$-equivariant linear maps from $V_\rho$ to $V$.

For $q\in Q$, recall that $T^{[\rho]}_q: V_\rho\to V_{q([\rho])}$ is the isomorphism between the representations $V_\rho$ and $V_{q([\rho])}$.  Define an isomorphism $T^\vee  _{q, [\rho]}$  from $\operatorname{Hom}^{G_\mathfrak{t}}(V_{q([\rho])}, V)$ to $\operatorname{Hom}^{G_\mathfrak{t}}(V_{\rho}, V)$ by
\begin{equation}\label{eq:dfn-tv}
T^\vee  _{q, [\rho]}(\phi):= \pi(s(q))\circ \phi \circ T^{[\rho]}_q.
\end{equation}
By a direct computation, we have the following property of $T^\vee_{q, [\rho]}$.
\begin{lem}\label{lem:twisted-rep}
\[
 T^\vee_{q_1, [\rho]}\circ T^\vee_{q_2, q_1([\rho])}=c(q_1,q_2) T^\vee_{q_1q_2, [\rho]}.
\]
\end{lem}
\begin{proof}
Let $\phi\in \operatorname{Hom}^{G_\mathfrak{t}}(V_{q_2q_1([\rho])}, V)$. Compute $ T^\vee_{q_1, [\rho]}\circ T^\vee_{q_2, q_1([\rho])}(\phi)$ to be
\begin{align*}
\pi(s(q_1))\pi(s(q_2))\circ \phi \circ T_{q_2}^{q_1([\rho])} \circ T_{q_1}^{[\rho]}
&=\pi(s(q_1)s(q_2))\circ \phi \circ c^{[\rho]}(q_1, q_2)T^{[\rho]}_{q_1q_2}\circ \rho(\tau(q_1, q_2))^{-1}\sigma(q_1, q_2)^{-1}\\
&= c^{[\rho]}(q_1, q_2)\pi(\tau(q_1, q_2))\pi(s(q_1q_2)) \circ \phi \circ T^{[\rho]}_{q_1q_2} \circ \rho(\tau(q_1, q_2))^{-1}\\
&=c^{[\rho]}(q_1, q_2)\pi(\tau(q_1, q_2)) T^\vee_{q_1q_2, [\rho]}(\phi)\circ \rho(\tau(q_1, q_2))^{-1}\\
&=c^{[\rho]}(q_1, q_2)T^\vee_{q_1q_2, [\rho]}(\phi).
\end{align*}
\end{proof}

Define a sheaf $\widehat{V}$ on $\widehat{G}^{\mathfrak{t}}$ by
\[
\widehat{V}|_{[\rho]}:= \operatorname{Hom}^{G_\mathfrak{t}}(V_\rho, V).
\]
By Lemma \ref{lem:twisted-rep}, $\{T^\vee_{q, [\rho]}\}$ makes $\widehat{V}$ into a $c$-twisted sheaf over $\widehat{G}^\mathfrak{t}\rtimes Q$. This defines a functor
\[
\cals: \mathfrak{Rep}_\mathfrak{t}(H)\longrightarrow \mathfrak{Rep}(\widehat{G}^\mathfrak{t}\rtimes Q, c).
\]

In the other direction, given a $c$-twisted sheaf $W$ over $\widehat{G}^\mathfrak{t}\rtimes Q$, define a vector space
\[
\widetilde{W}=\bigoplus_{[\rho]\in \widehat{G}^\mathfrak{t}} W|_{[\rho]}\otimes V_\rho.
\]
For $h=(s, g, q)\in H_{\mathfrak{t}}$, define an action of $H_\mathfrak{t}$ on $\widetilde{W}$ by the following formula,
\[
\pi(h)=\sum_{[\rho]\in \widehat{G}^\mathfrak{t}} T^\vee_{q, [\rho]}\otimes \rho(s, g)\circ {T^{[\rho]}_q}^{-1}.
\]
\begin{lem}\label{lem:h-rep}
\[
\pi(h_1)\circ \pi(h_2)=\pi(h_1h_2).
\]
\end{lem}
\begin{proof}
We compute $\pi(h_1)\pi(h_2)$ by
\begin{eqnarray*}
&&\sum_{[\rho_1]} T^\vee_{q_1, [\rho]}\otimes \rho_1(s_1, g_1){T^{[\rho]}_{q_1}}^{-1}\sum_{[\rho_2]} T^\vee_{q_2, [\rho_2]}\otimes \rho_2(s_2, g_2){T^{[\rho_2]}_{q_2}}^{-1}\\
&=&\sum_{[\rho_1]}T^\vee_{q_1, [\rho]}T^\vee_{q_2, q_1([\rho])}\otimes  \rho(s_1, g_1){T^{[\rho]}_{q_1}}^{-1}\circ q_1([\rho])(s_2, g_2) {T^{q_1([\rho])}_{q_2}}^{-1}\\
&=&\sum_{[\rho]}c^{[\rho]}(q_1, q_2)T^\vee_{q_1q_2, [\rho]}\otimes \rho(s_1, g_1){T^{[\rho]}_{q_1}}^{-1}\circ q_1([\rho])(s_2, g_2){T^{[\rho]}_{q_1}}\circ {T^{[\rho]}_{q_1}}^{-1} \circ {T^{q_1([\rho])}_{q_2}}^{-1}\\
&=&\sum_{[\rho]}c^{[\rho]}(q_1, q_2)T^\vee_{q_1q_2, [\rho]}\otimes \rho(s_1, g_1)\rho(\ad_{q_1}(s_2, g_2))\circ{\sigma(q_1, q_2) \rho(\tau(q_1, q_2))T^{[\rho]}_{q_1q_2}}^{-1}c^{[\rho]}(q_1, q_2)^{-1}\\
&=&\sum_{[\rho]}T^\vee_{q_1q_2, [\rho]}\otimes \rho\Big( (s_1, g_1)\ad_{q_1}(s_2, g_2)\big(\sigma(q_1, q_2), \tau(q_1, q_2)\big)\Big){T^{[\rho]}_{q_1q_2}}^{-1}\\
&=&\pi(h_1h_2).
\end{eqnarray*}
\end{proof}
 Lemma \ref{lem:h-rep} shows that $(\widetilde{W}, \pi)$ is a representation of $H_\mathfrak{t}$ such that $S^1$ acts with weight $1$. Therefore, we have defined a functor
 \[
 \calt: \mathfrak{Rep}(\widehat{G}^\mathfrak{t}\rtimes Q, c)\longrightarrow \mathfrak{Rep}_\mathfrak{t}(H).
 \]
 It is straightforward to check that
 \[
 \calt\circ \cals =id\qquad\text{and}\qquad \cals \circ \calt=id.
 \]
Hence, we conclude with the following theorem.
\begin{thm}\label{prop:category-gp} The categories $\mathfrak{Rep}(\widehat{G}^\mathfrak{t}\rtimes Q, c)$ and $\mathfrak{Rep}_\mathfrak{t}(H)$ are equivalent.
\end{thm}
\section{$S^1$-gerbe on an orbifold}\label{sec:orbifold}
In this section, we generalize the discussion in Sec. \ref{sec:group} to study a central extension of a general orbifold.

\subsection{$S^1$-gerbe over an orbifold} Let $\mathfrak{Q}\rightrightarrows M$ be a proper \'etale Lie groupoid representing an orbifold $\B$. We consider a $G$-gerbe $\Y$ over $\B$, which by \cite{la-st-xu} can be presented by an extension of $\mathfrak{Q}$ by a bundle of groups that are isomorphic to $G$, i.e.,
\[
G\times M\longrightarrow \mathfrak{H}\longrightarrow \mathfrak{Q}.
\]

We are interested in an $S^1$-gerbe $\Y_\mathfrak{t}$ over an orbifold $\Y$.  More precisely, $\Y_\mathfrak{t}$ is presented as an $S^1$-central extension of the groupoid $\mathfrak{H}$, i.e.,
\[
 S^1\times M\longrightarrow \mathfrak{H}_\mathfrak{t}\longrightarrow \mathfrak{H}.
\]
By the assumption that $\mathfrak{H}_\mathfrak{t}$ and its nerve spaces are disjoint unions of contractible open charts, the groupoid $\mathfrak{H}_\mathfrak{t}$ can be written as $S^1\times \mathfrak{H}$ with the product defined by an $S^1$-valued 2-cocycle $\mathfrak{t}$ on $\mathfrak{H}$.

Consider the restriction of $\mathfrak{t}$ to the subgroupoid $G\times M$. This defines an $S^1$ central extension
\[
S^1\times M\longrightarrow G\times_\mathfrak{t}M \longrightarrow G\times M.
\]
It is not hard to see that the groupoid $G\times_\mathfrak{t}M\rightrightarrows M$ is a normal subgroupoid of $\mathfrak{H}_\mathfrak{t}$ with the quotient being $\mathfrak{Q}$, i.e.,
\begin{equation}\label{eq:ext-gpd}
G\times_\mathfrak{t}M\longrightarrow \mathfrak{H}_\mathfrak{t}\longrightarrow \mathfrak{Q}.
\end{equation}
\begin{lem}\label{lem:coboundary} For each $x_0\in M$, there is a neighborhood $U$ of $x_0$, and a smooth function $\phi:U\times G\to  S^1$ such that for every pair $g, h\in G$
\[
\phi(x, g)^{-1}\phi(x,h)^{-1}\phi(x, gh) \mathfrak{t}_x(g, h)
\]
is constant over $U$.
\end{lem}
\begin{proof}
Without loss of generality, we assume that $\mathfrak{t}_{x_0}$ is a constant function with value 1 on $G\times G$.  Consider the exponential map $\exp: \reals\to S^1$, which is a local diffeomorphism. Using this property and the fact that $G$ is finite, we can choose a small neighborhood $U$ of $x_0$, such that over $U$ there is a  smooth function $\bar{\mathfrak{t}}: U\times G\times G\to \reals$, such that $\exp(2\pi \sqrt{-1}\bar{\mathfrak{t}})$ is equal to $\mathfrak{t}$ and $\bar{\mathfrak{t}}_{x_0}=0$.

Let $d$ be the differential of the group cochain complex of $G$. As $\mathfrak{t}$ is a cocycle and
\[
\exp(2\pi \sqrt{-1}d(\bar{\mathfrak{t}}_x))=d (\mathfrak{t}_x)=1,
\]
the coboundary $d(\bar{\mathfrak{t}}_x)$ must be $\integers$ valued.  At $x_0$, since $\bar{\mathfrak{t}}_{x_0}$ is assumed to be the zero function on $G\times G$, $d(\bar{\mathfrak{t}}_{x_0})$ is equal to the zero function on $G\times G\times G$. Hence, we conclude that $d(\bar{\mathfrak{t}}_x)$ is a zero function on $G\times G\times G$ for every $x$ as $\bar{t}$ is assumed to be a smooth function.  Therefore we conclude that $\bar{\mathfrak{t}}_x$ is a 2-cocycle on $G$.

Because $G$ is finite, $H^2(G, \reals)$ is zero. More precisely, for every 2-cocycle $F$ on $G$,  the function
\[
f(g):= \frac{1}{|G|}\sum_{x\in G}F(x, g)
\]
satisfies $df=F$. This construction defines a smooth function $\Phi:U\times G\to \reals$ such that $d(\Phi_x)=\bar{\mathfrak{t}}_x$.  Define $\phi_x=\exp(2\pi \sqrt{-1}\Phi_x): U\times G\to S^1$. Then we can directly check that
\[
d\phi_x=\mathfrak{t}_x,
\]
which satisfies the desired property.
\end{proof}

By Lemma \ref{lem:coboundary}, we choose an open cover of $M$ such that every open set is contained in a neighborhood introduced in that lemma.  We replace the groupoid extension (\ref{eq:ext-gpd}) by a Morita equivalent one via pulling back the involved groupoids to the open cover above. By \cite{mo-pr-ind}, we can  furthermore assume that, $\mathfrak{Q}$ and all its nerve spaces are disjoint unions of contractible open sets, by pulling back all the groupoids in (\ref{eq:ext-gpd}) to a finer cover of $M$. Also $\mathfrak{H}$ and its nerve spaces are disjoint unions of contractible open charts. Hence we can choose a smooth section $\alpha: \mathfrak{Q}\to \mathfrak{H}_\mathfrak{t}$.  Using $\alpha$, we can rewrite $\mathfrak{H}_\mathfrak{t}$ as in Sec. \ref{sec:gp-rep}. More explicitly, we have the following data.
\begin{enumerate}
\item A smooth map $\ad: \mathfrak{Q}\to \aut(G\times_\mathfrak{t}M)$, i.e., for $q\in \mathfrak{Q}$, $\ad_q$ is an isomorphism from $G\times_{\mathfrak{t}}M|_{t(q)}$ to $G\times_{\mathfrak{t}}M|_{s(q)}$ defined by
\[
\ad_q(t,g)=\alpha(q)(t,g)\alpha(q)^{-1}.
\]
\item A smooth map $\bar{\tau}$ from $\mathfrak{Q}^{(2)}:=\{(q_1, q_2)\in \mathfrak{Q}\times \mathfrak{Q}|t(q_1)=s(q_2)\}$ to $G\times_{\mathfrak{t}}M$ defined by
\[
\bar{\tau}(q_1, q_2)\in G\times_{\mathfrak{t}}M|_{s(q_1)},
\]
such that
\[
\alpha(q_1)\alpha(q_2)=\bar{\tau}(q_1, q_2)\alpha(q_1q_2).
\]
\end{enumerate}

By Lemma \ref{lem:coboundary}, there is a smooth function $\phi: M\times G\to S^1$ such that the following map
\[
\Phi: (s, g, x)\mapsto (s\phi(x, g), g, x)
\]
maps $G\times_{\mathfrak{t}}M$ isomorphically onto a new groupoid such that the new cocycle $\mathfrak{t}'$ is locally constant.  Therefore, we can restrict ourselves to the consideration of an extension of the form (\ref{eq:ext-gpd}) such that the cocycle $\mathfrak{t}$ on $G\times_{\mathfrak{t}}M$ is locally constant.

\begin{rmk}In the above discussion, we have used the smooth function $\phi$ in Lemma \ref{lem:coboundary}. Such a function, in general, is not unique. However, different choices of $\phi$ lead to isomorphic groupoid extensions.
\end{rmk}

\subsection{Dual of an $S^1$-gerbe} \label{subsec:gpd-dual}In this subsection we construct a proper \'etale groupoid with an $S^1$-gerbe, associated to the groupoid extension (\ref{eq:ext-gpd}).
On every connected component $U$ of $M$, $\mathfrak{t}|_{G\times U}$ is constant.  Recall that $\widehat{G}^\mathfrak{t}$ is the set of isomorphism classes of irreducible $\mathfrak{t}$-twisted unitary representations of $G$. Consider $\widehat{G}^\mathfrak{t}\times U$, and let $\widehat{M}$ be the disjoint union of $\widehat{G}^\mathfrak{t}\times U$; denote by $\lambda$ the natural map $\lambda: \widehat{M}\to M$.

For every $q\in \mathfrak{Q}$,  $\ad_q$ defines an isomorphism from $G\times_{\mathfrak{t}}M|_{t(q)}$ to $G\times_\mathfrak{t}M|_{s(q)}$. As is explained in Sec. \ref{subsec:group-algebra}, $\ad_q$ defines an action map from $\widehat{G}^{\mathfrak{t}(s(q))}$ to $\widehat{G}^{\mathfrak{t}(t(q))}$, and therefore defines a $\mathfrak{Q}$ action on $\widehat{M}$. We consider the transformation groupoid $\widehat{M}\rtimes \mathfrak{Q}$.

For every point $([\rho], x)\in \widehat{M}$, we note that  $[\rho]$ is an isomorphism class of an irreducible $\mathfrak{t}|_{G\times \{x\}}$-twisted unitary representation of $G$. For every $[\rho]$, choose a representative $(V_{\rho}, \rho:G \to \mathfrak{U}(V_\rho))$, an irreducible $\mathfrak{t}|_{G\times \{x\}}$-twisted unitary representation of $G$. As the cocycle $\mathfrak{t}$ on $G\times_\mathfrak{t}M$ is locally constant, the representative $(V_\rho, \rho)$ can be chosen to be a locally trivial vector bundle on $\widehat{M}$. Similarly to the construction in Sec. \ref{subsec:group-algebra}, for every $([\rho], q)\in \widehat{M}\rtimes \mathfrak{Q}$, there is an isomorphism $T^{[\rho]}_q: V_{\rho}|_{s(q)}\to V_{q([\rho])}|_{t(q)}$, such that
\[
\rho(\ad_q(t, g))={T^{[\rho]}_q}^{-1}\circ q([\rho])(t, g)\circ T^{[\rho]}_q.
\]
As $G$ is finite, $T^{[\rho]}_{q}$ can be chosen to be a locally constant map. It is straightforward to check that
\[
c^{[\rho]}(q_1, q_2)
:=T^{q_1([\rho])}_{q_2}\circ T^{[\rho]}_{q_1}\rho(\bar{\tau}(q_1, q_2)){T^{[\rho]}_{q_1q_2}}^{-1}
\]
is a linear map on $V_{q_1q_2([\rho])}$ commuting with the $G$-action. Schur's Lemma implies that $c^{[\rho]}(q_1, q_2)$ must be a scalar of norm 1. By exactly the same arguments as \cite[Prop. 4.5]{tt}, we conclude that $c$ is a smooth 2-cocycle on $\widehat{M}\rtimes \mathfrak{Q}$.  Let us denote by $\widehat{\Y}$ the orbifold associated to the $\widehat{M}\rtimes \mathfrak{Q}$. On $\widehat{\Y}$, there is also an $S^1$-gerbe $\widehat{\Y}_c$ defined by the $2$-cocycle $c$.

\begin{rmk}\label{rmk:flat}
Differently from what was obtained in \cite[Prop. 4.6]{tt}, the above cocycle $c$ on $\widehat{M}\rtimes \mathfrak{Q}$ is usually not locally constant. For example, when $G$ is trivial, $\widehat{M}\rtimes \mathfrak{Q}$ is $\mathfrak{Q}$, and $c=\mathfrak{t}$. Nevertheless, when $\mathfrak{t}$ is locally constant, then one can check that $c$ is locally constant.
\end{rmk}

\begin{ex}
Consider an extension of  a proper \'etale groupoid $\mathfrak{Q}$, i.e.,
\[
1\rightarrow G\times M\rightarrow \mathfrak{H}\rightarrow \mathfrak{Q}\rightarrow 1.
\]
Let $t_0$ be an $S^1$-valued 2-cocycle on $\mathfrak{Q}$, and $\mathfrak{t}$ be the pullback of $t_0$ to $\mathfrak{H}$. The dual groupoid of $\mathfrak{H}_\mathfrak{t}$, introduced in this subsection,  is $\widehat{G}\rtimes \mathfrak{Q}$. In \cite{tt}, an $S^1$-valued 2-cocycle $c_0$ on the transformation groupoid $\widehat{G}\rtimes \mathfrak{Q}$ was introduced. There is a canonical groupoid morphism from $\widehat{G}\rtimes \mathfrak{Q}$ to $\mathfrak{Q}$. The pullback of $t_0$  defines an $S^1$-valued $2$-cocycle $\bar{t}$ on $\widehat{G}\rtimes \mathfrak{Q}$. The cocycle $c$ on $\widehat{G}\rtimes \mathfrak{Q}$ that is dual to $\mathfrak{H}_\mathfrak{t}$, as discussed in this subsection, is $c_0\bar{t}$.  It is easy to check that $t$ is locally constant if and only if $c$ is locally constant.
\end{ex}

\subsection{Categories of sheaves}

In this subsection, we discuss the connection between the orbifolds $(\Y, \mathfrak{t})$ and $(\widehat{\Y}, c)$.  We prove that as noncommutative algebraic spaces the two spaces $(\Y, \mathfrak{t})$ and $(\widehat{\Y}, c)$ are isomorphic.

Over $\widehat{M}$, define a sheaf $\widehat{\calv}$, a vector bundle on each component, as follows. For each $([\rho], x)\in \widehat{M}$, $\widehat{\calv}|_{([\rho], x)}$ is the representation space $V_{\rho}$ chosen in Sec. \ref{subsec:gpd-dual} for the group $G_{\mathfrak{t}_x}$. As $\mathfrak{t}|_{G\times_{\mathfrak{t}}M}$ is assumed to be locally constant, $\widehat{\calv}$ is a smooth vector bundle on each component of $\widehat{M}$.
By the natural homomorphism from $\mathfrak{H}_\mathfrak{t}$ to $\mathfrak{Q}$, we have that $\mathfrak{H}$ acts on $\widehat{M}$. And there is a natural groupoid morphism from $\widehat{M}\rtimes \mathfrak{H}_\mathfrak{t}$ to $\widehat{M}\rtimes \mathfrak{Q}$. Pulling back the cocycle $c$ on $\widehat{M}\rtimes \mathfrak{Q}$ defines an $S^1$-valued 2-cocycle on $\widehat{M}\rtimes \mathfrak{H}_\mathfrak{t}$, which is again denoted by $c$.

We equip $\widehat{\calv}$ with a left $c^{-1}$-twisted $\widehat{M}\rtimes \mathfrak{H}_\mathfrak{t}$ action $\alpha_{\widehat{M}\rtimes \mathfrak{H}_\mathfrak{t}}$ as follows. For $h\in \mathfrak{H}_\mathfrak{t}$, by the section $\alpha: \mathfrak{Q}\to \mathfrak{H}_\mathfrak{t}$, we can write $h=g\alpha(q)$ where $q$ is the image of $h$ under the canonical map from $\mathfrak{H}_\mathfrak{t}$ to $\mathfrak{Q}$, and $g\in G\times_\mathfrak{t}M$. Let $(q([\rho]), x)\in \widehat{G}\times_{\mathfrak{t}}M$ with $x=t(h)$ and $\xi\in \widehat{\calv}_{q([\rho])}$. The computation of Lemma \ref{lem:h-rep} shows that
\[
\alpha_{\widehat{M}\rtimes \mathfrak{H}_\mathfrak{t}}(\xi):= \rho(g)\circ {T^{[\rho]}_q}^{-1}(\xi)\in \widehat{\calv}_{\rho}
\]
defines a natural $c^{-1}$-twisted action of $\widehat{M}\rtimes \mathfrak{H}_\mathfrak{t}$ on $\widehat{\calv}$.   Therefore, we have defined a $c^{-1}$-twisted sheaf $\widehat{\calv}$ on $\widehat{M}\rtimes \mathfrak{H}_\mathfrak{t}$, where the subgroupoid $S^1\times M$ acts with weight $1$.

Let $\mathfrak{Rep}_\mathfrak{t}(\mathfrak{H})$ be the category of (coherent) $\mathfrak{H}_\mathfrak{t}$-sheaves of vector spaces over $\complex$ such that the subgroupoid $S^1\times M$ acts with weight $1$. Let $\mathfrak{Rep}(\widehat{M}\rtimes \mathfrak{Q}, c)$ be the category of $c$-twisted (coherent)  $\widehat{M}\rtimes \mathfrak{Q}$-sheaves of vector spaces over $\complex$.  The following theorem is a generalization of Theorem \ref{prop:category-gp}.
\begin{thm}\label{thm:sheaves-eq}The two categories $\mathfrak{Rep}_\mathfrak{t}(\mathfrak{H})$ and $\mathfrak{Rep}(\widehat{M}\rtimes \mathfrak{Q}, c)$ are isomorphic.
\end{thm}

\begin{proof}
Let $\calw$ be a (coherent) $\mathfrak{H}_\mathfrak{t}$-sheaf such that the subgroupoid $S^1\times M$ acts with weight $1$. Pull back $\calw$ to $\widehat{M}$ via the canonical map $\lambda: \widehat{M}\to M$. Then $\lambda^*(\calw)$ is a $\widehat{M}\rtimes \mathfrak{H}_\mathfrak{t}$-sheaf such that the subgroupoid $S^1\times M$ acts with weight 1.

Denote by $G\times \widehat{M}\rightrightarrows \widehat{M}$ the groupoid of the trivial bundle of the group $G$ over $\widehat{M}$. Let $G\times_\mathfrak{t}\widehat{M}\rightrightarrows \widehat{M}$ be the groupoid defined by pulling back the cocycle $\mathfrak{t}$ on $G\times M\rightrightarrows M $ along the natural groupoid morphism $\pi:G\times \widehat{M}\to G\times M$. It is not difficult to see that $G\times_\mathfrak{t}\widehat{M}$ is a normal subgroupoid of $\widehat{M}\rtimes \mathfrak{H}_\mathfrak{t}$ and the cocycle $c^{-1}$ restricts to a trivial one on  $G\times_\mathfrak{t}\widehat{M}$. Hence $\widehat{\calv}$ is a $G\times_\mathfrak{t}\widehat{M}$-sheaf.  Define a sheaf $\widehat{\calw}$ on $\widehat{M}$ by
\[
\widehat{\calw}:= \mathcal{H}om^{G\times _{\mathfrak{t}}\widehat{M}}(\widehat{\calv}, \lambda^*\calw)
\]

The same formula (\ref{eq:dfn-tv}) for $T^\vee_{q, [\rho]}$ in Sec. \ref{sec:gp-rep} and the computation in Lemma \ref{lem:twisted-rep}, equip $\widehat{\calw}$ with a $c$-twisted $\widehat{M}\rtimes \mathfrak{Q}$-sheaf structure.

Define a functor $\cals: \mathfrak{Rep}_\mathfrak{t}(\mathfrak{H})\to \mathfrak{Rep}(\widehat{M}\rtimes \mathfrak{Q}, c)$ by
\[
\cals(\calw):=\widehat{\calw}.
\]

We next define a functor $\calt:\mathfrak{Rep}(\widehat{M}\rtimes \mathfrak{Q}, c)\to \mathfrak{Rep}_\mathfrak{t}(\mathfrak{H})$ as follows. Let $\widetilde{\calw} $ be a $c$-twisted $\widehat{M}\rtimes\mathfrak{Q}$-sheaf, which can be viewed as a $c$-twisted $\widehat{M}\rtimes \mathfrak{H}_\mathfrak{t}$-sheaf by the canonical map from $\mathfrak{H}_\mathfrak{t}$ to $\mathfrak{Q}$. As $\widehat{\calv}$ is a $c^{-1}$-twisted $\widehat{M}\rtimes {\mathfrak{H}}_\mathfrak{t}$-sheaf, the tensor product $\widetilde{\calw}\otimes \widehat{\calv}$ is a $\widehat{M}\rtimes \mathfrak{H}_\mathfrak{t}$-sheaf without any twist. Hence $\pi_*(\widetilde{\calw}\otimes \widehat{\calv})$ is an $\mathfrak{H}_\mathfrak{t}$-sheaf such that the subgroupoid $S^1\times M$  acts with weight $1$.

Define a functor $\calt: \mathfrak{Rep}(\widehat{M}\rtimes \mathfrak{Q}, c)\to \mathfrak{Rep}_\mathfrak{t}(\mathfrak{H})$ by
\[
\calt(\widetilde{\calw})=\pi_*(\widetilde{\calw}\otimes \widehat{\calv}).
\]

With the definition of $\cals$ and $\calt$, the proof that they are inverse to each other reduces to the local computation that is explained Sec. \ref{sec:gp-rep}. We omit the details.
\end{proof}

Considering the Grothendieck groups of the categories, we obtain the following corollary.
\begin{cor}
$K_0(\Y, \mathfrak{t})\cong K_0(\widehat{\Y}, c)$.
\end{cor}

\subsection{Twisted groupoid algebra}In this subsection, we prove that as noncommutative differential geometric spaces $(\Y, \mathfrak{t})$ and $(\widehat{\Y}, c)$ are isomorphic. For $\mathfrak{H}_\mathfrak{t}$, we consider the space $C^\infty_c(\mathfrak{H})$ of compactly supported smooth functions on $\mathfrak{H}$ with the following $\mathfrak{t}$-twisted convolution product. For $\phi, \varphi \in C^\infty_c(\mathfrak{H})$,
\[
\phi\circ_{\mathfrak{t}}\varphi(h):= \sum_{t(h)=t(h')}\mathfrak{t}(hh'^{-1}, h')\phi(hh'^{-1})\varphi(h'),
\]
where $S^1$ is naturally embedded in $\mathbb{C}$ as the unit circle.  Similarly, for $(\widehat{M}\rtimes \mathfrak{Q}, c)$, we consider the space of $C^\infty_c(\widehat{M}\rtimes \mathfrak{Q})$ of compactly supported smooth functions on $\widehat{M}\rtimes \mathfrak{Q}$ with a similar $c$-twisted convolution product defined in the same way. Both $(C^\infty_c(\mathfrak{H}), \circ_\mathfrak{t})$ and $(C^\infty_c(\widehat{M}\rtimes \mathfrak{Q}), \circ_c)$ are Fr\'echet algebras, and therefore bornological algebras \cite{me}, \cite{pptt}. The following theorem is a generalization of \cite[Thm. 4.8]{tt}.

\begin{thm}
\label{thm:gpd-morita} The two twisted groupoid algebras $(C^\infty_c(\mathfrak{H}), \circ_\mathfrak{t})$ and $(C^\infty_c(\widehat{M}\rtimes \mathfrak{Q}), \circ_c)$ are Morita equivalent as bornological algebras.
\end{thm}

\begin{proof}
We explain the construction of the Morita equivalence bimodule between the algebras $(C^\infty_c(\mathfrak{H}), \circ_\mathfrak{t})$ and $(C^\infty_c(\widehat{M}\rtimes \mathfrak{Q}), \circ_c)$.

We pull back the sheaf $\widehat{\calv}$ on $\widehat{M}$ to a sheaf $\widetilde{\calv}$ on $\widehat{M}\rtimes \mathfrak {Q}$ via the source map $s$ from $\widehat{M}\rtimes \mathfrak{Q}$ to $\widehat{M}$.  Composing $s$ with the projection map $\pi$ from $\widehat{M}$ to $M$ gives a map $j: \widehat{M}\rtimes \mathfrak{Q}\to M$.  Via the homomorphism from $\mathfrak{H}_\mathfrak{t}$ to $\mathfrak{Q}$, the groupoid $\mathfrak{H}_\mathfrak{t}$ naturally acts on $\widehat{M}\rtimes \mathfrak{Q}$.  Hence, $\widetilde{\calv}$ is equipped with an $\mathfrak{H}_\mathfrak{t}$ action as follows. Consider $([\rho],q)\in \widehat{M}\rtimes \mathfrak{Q}$, and $(\lambda, g, q')\in \mathfrak{H}_t$ such that $t(q')=s(q)$. For $\xi \in \widetilde{\calv}|_{([\rho], q)}$, we set $(\lambda, g, q')\cdot \xi$ to be an element in $\widetilde{\calv}_{(q'([\rho]), q'q)}$ defined by
\[
c(q',q)(\lambda, g, q')\cdot \xi,
\]
where $\xi\in \widetilde{\calv}|_{([\rho], q)}=s^* \widehat{\calv}|_{([\rho], s(q))}$ is mapped to $(\lambda, g, q')\cdot \xi\in \widetilde{\calv}|_{(q'([\rho]), q'q)}=s^*\widehat{ \calv}|_{(q'([\rho]), s(q'))}$ via the $c^{-1}$-twisted $\widehat{M}\rtimes \mathfrak{H}_\mathfrak{t}$-sheaf structure on $\calv$.  This makes the space of sections $\Gamma(\widetilde{\calv})$ of $\widetilde{\calv}$ into a left $\mathfrak{H}_\mathfrak{t}$-module and therefore also a left $(C^\infty(\mathfrak{H}), \circ_\mathfrak{t})$-module.

On the other hand, the groupoid $\widehat{M}\rtimes \mathfrak{Q}$ acts on $\widehat{M}\rtimes \mathfrak{Q}$ and therefore also on $\widetilde{\calv}$ by right translation. Therefore, the space $\Gamma(\widetilde{\calv})$ is a right $c$-twisted $\widehat{M}\rtimes\mathfrak{Q}$-module and therefore a right $(C^\infty_c(\widehat{M}\rtimes \mathfrak{Q}), \circ_c)$-module.

It is not hard to see that the left $\mathfrak{H}_\mathfrak{t}$ action on $\Gamma(\widetilde{\calv})$ commutes with the right $\widehat{M}\rtimes \mathfrak{Q}$ action. And therefore $\Gamma(\widetilde{\calv})$ is a left $(C^\infty_c(\mathfrak{H}), \circ_\mathfrak{t})$ and right $(C^\infty_c(\widehat{M}\rtimes \mathfrak{Q}), \circ_c)$ bimodule.  One can check that $\Gamma(\widetilde{\calv})$ is a Morita equivalence bimodule by reducing it to a small neighborhood and applying Theorem \ref{thm:gp-morita}. We refer the reader to \cite[Thm. 4.8]{tt} for more details.
\end{proof}

Cyclic homology and Hochschild homology are invariant under Morita equivalence. As a corollary of Theorem \ref{thm:gpd-morita}, we have the following result
\begin{cor}
\begin{eqnarray*}
&&HH_\bullet((C^\infty_c(\mathfrak{H}), \circ_\mathfrak{t}))\cong HH_\bullet((C^\infty_c(\widehat{M}\rtimes \mathfrak{Q}), \circ_c))\\
&&HC_\bullet((C^\infty_c(\mathfrak{H}), \circ_\mathfrak{t})) \cong HC_\bullet((C^\infty_c(\widehat{M}\rtimes \mathfrak{Q}), \circ_c))\\
\end{eqnarray*}
\end{cor}

In \cite{tu-xu}, the (periodic) cyclic homology of an $S^1$ gerbe $\mathfrak{X}_c$ over an orbifold $\mathfrak{X}$ is computed to be the compactly supported twisted cohomology groups $H^\bullet_{\rm cpt}(\mathfrak {X}, c)$ introduced in \cite[Def. 3.10]{tu-xu}. So we conclude this section with the following corollary.
\begin{cor}
\[
H^\bullet_{\rm cpt}(\Y, \mathfrak{t})=H^\bullet_{\rm cpt}(\widehat{\Y}, c), \qquad \bullet=\text{ even, odd}.
\]
\end{cor}
\section{Appendix}
A language equivalent to that of Lie groupoids used above is that of differentiable stacks.  While perhaps more abstract and technically demanding it does offer a more conceptual, if less transparent, point of view.  In the process of writing this text the authors have faced a choice of selecting a language to use and have decided against stacks.  However this appendix is provided to outline some of the main ideas of the paper, without proper stack theoretic justification, to those more comfortable with this point of view.  To use an analogy, in the Appendix we deal with principal bundles, whereas in the main text we consider the associated gluing data. 

Recall that a $G$-bitorsor $X$ is a set with commuting left and right $G$ actions such that both actions are free, proper, and transitive. A morphism between two $G$-bitorsors $X_1$ and $X_2$ is a map $f:X_1\to X_2$ satisfying
\[
f(g\cdot x \cdot g')=g\cdot f(x)\cdot g'.
\]
Let $\bit(G)$ be the collection of all $G$-bitorsors. There is a natural  associative product on $\bit(G)$ defined as follows. For $X_1, X_2\in \bit(G)$, consider the diagonal $G$ action on $X_1 \times X_2$ by
\[
(g, x_1, x_2)\mapsto (x_1\cdot g^{-1}, g\cdot x_2).
\]
Define  $X_1\times_G X_2$ to be the quotient of $X_1\times X_2$ by the above $G$ action. It is straightforward to check that the left $G$ action on $X_1$ and the right $G$ action on $X_2$ make $X_1\times_G X_2$ into a $G$-bitorsor.  The group $G$ with the left and right translation $G$-action is the identity of the product.  Furthermore, if $f_i:X_i\to Y_i$, $i=1,2$ are morphisms of $G$-bitorsors, $f_1\times f_2:X_1\times_G Y_1\to X_2\times_G Y_2$ is defined to be $f_1\times f_2(x_1, x_2)=[(f_1(x_1),f_2(x_2))]\in Y_1\times_G Y_2$.  The inverse of $X$ is the same as a set, with the right and left actions interchanged.  Accordingly, $(\bit(G), \times_G)$ is a $2$-group; it should be viewed as a categorification of the outer automorphism group of $G$.    It can be made into a differentiable stack in an obvious manner.  Note that its Lie groupoid model is $[Aut(G)/G]$.

Let $\mathfrak{Rep}(G)$ be the category of  finite dimensional $G$-representations over $\mathbb{C}$ with morphisms being $G$-equivariant linear maps.  The $2$-group $\bit(G)$ acts on $\mathfrak{Rep}(G)$ and on the full subcategory $\mathfrak{IRep}(G)$ of irreducible representations as follows: $X\cdot V=X\times_G V$, where
\[
 X\times_G V:=X\times V/(xg,v)\sim (x, gv).
\]

Let $\widehat{G}$ be the set of isomorphism classes of irreducible $G$-representations. Equip $\widehat{G}$ with the discrete topology. We observe that the $\bit(G)$ action on $\mathfrak{IRep}(G)$ maps isomorphic irreducible $G$-representations to isomorphic ones. Hence $\bit(G)$ also descends to an action on $\widehat{G}$.  The natural forgetful map $\pi:\mathfrak{IRep}(G)\to \widehat{G}$ has an additional important bit of structure.  Namely, consider the category $\call$ of one dimensional vector spaces over $\mathbb{C}$. As the tensor product of lines is still a line, the tensor product makes $\call$ into a 2-group, with duality serving as inversion. Then $\call$ acts on the fibers of $\pi$ via the tensor product, and furthermore $\pi$ has the structure of a $\bit(G)$-equivariant principal $\call$ bundle.  Those who are less comfortable with $2$-groups will benefit from thinking about plain groups, here and below.  In particular when considering a principal $A$-bundle $\calp$ with $A$ a $2$-group, and an $A$-module category $C$, one may just as with groups, perform the fiber replacement construction $\calp\times_A C$.

\subsection{Untwisted case}
This section deals with the differentiable stack version of groupoid extensions that were used in the main body of the paper.  More precisely,  that is what the $\bit(G)$-principal bundle below corresponds to.

Let $\tor(G)$ denote $G$-torsors.  Then $\bit(G)$ obviously acts on $\tor(G)$ and in fact is $Aut(\tor(G))$. By a $G$-gerbe on a differentiable stack $\mathcal{M}$ we mean a $\bit(G)$-principal bundle $\calp_B$ on $\mathcal{M}$.   Following \cite{la-st-xu}, we can present $\mathcal{M}$ by a Lie groupoid $\mathfrak{Q}\rightrightarrows M$, and a $G$-gerbe on $\mathcal{M}$ by a groupoid extension $\mathfrak{H}$ of $\mathfrak{Q}$ by a bundle of groups over $M$. It is perhaps more natural to consider the equivalent data of $$\calp=\calp_B\times_{\bit(G)}\tor(G)$$ over $\calm$, as the gerbe itself.  The analogy is the difference between a bundle of frames and a vector bundle. Note that the group $G$ itself can vary over $\calm$, within reason.

Then the dual of $\calp$ relative to $\calm$ is $$\widehat{\calp}=\calp_B\times_{\bit(G)}\widehat{G}$$ that has the additional structure of a $\call$ principal bundle, i.e., an $S^1$-gerbe $$\widetilde{\calp}=\calp_B\times_{\bit(G)}\mathfrak{IRep}(G)$$ over it.

We can then compare the sheaves on $\calp$ to the $\widetilde{\calp}$ twisted sheaves on $\widehat{\calp}$.  More precisely, we claim an equivalence of categories $$\Gamma(\calp,\calp\times \text{Vect}_{\mathbb{C}})\simeq\Gamma(\widehat{\calp}, \widetilde{\calp}\times_{\call}\text{Vect}_{\mathbb{C}})$$ where $\text{Vect}_{\mathbb{C}}$ denotes the category of finite dimensional vector spaces over $\mathbb{C}$.  The equivalence follows, roughly, from the string of equivalences:
\begin{align*}
\Gamma(\calp,\calp\times \text{Vect}_{\mathbb{C}})&\simeq \Gamma(\calm, \calp_B\times_{\bit(G)}\mathfrak{Rep}(G))\\
&\simeq\Gamma(\calm, \calp_B\times_{\bit(G)}\Gamma(\widehat{G},\mathfrak{IRep}(G)\times_{\call}\text{Vect}_{\mathbb{C}}))\\
&\simeq\Gamma(\calp_B\times_{\bit(G)}\widehat{G}, \calp_B\times_{\bit(G)}(\mathfrak{IRep}(G)\times_{\call}\text{Vect}_{\mathbb{C}}))\\
&\simeq\Gamma(\widehat{\calp}, (\calp_B\times_{\bit(G)}\mathfrak{IRep}(G))\times_{\call}\text{Vect}_{\mathbb{C}})\\
&\simeq\Gamma(\widehat{\calp}, \widetilde{\calp}\times_{\call}\text{Vect}_{\mathbb{C}}).
\end{align*}

\subsection{Twisted case}
Here we deal with the differentiable stack version of $S^1$ central extensions of groupoid extensions.  Namely, we treat the case of twisted $G$-gerbes using considerations very similar to those of the untwisted case.  More precisely, let $G_\mathfrak{t}$ denote a central extension of $G$ by $S^1$.  Again, this data may in principal vary over the base.  Use a slight modification of the construction of $\bit(G)$ to define $\bit_\mathfrak{t}(G)$.  Namely, $\bit_\mathfrak{t}(G)$ consists of $G_\mathfrak{t}$-bitorsors $X$ such that $xc=cx$ for all $c\in S^1$ and any, and hence all, $x\in X$.

Denote by $\mathfrak{Rep}_\mathfrak{t}(G)$ the full subcategory of $\mathfrak{Rep}(G_\mathfrak{t})$ that consists of representations with $S^1$ acting with weight 1. Then as before we have a natural action of  $\bit_\mathfrak{t}(G)$ on $\mathfrak{Rep}_\mathfrak{t}(G)$ and on its full subcategory $\mathfrak{IRep}_\mathfrak{t}(G)$.  We denote the isomorphism classes of the latter by $\widehat{G}^\mathfrak{t}$.

Note that we have a natural $2$-group map from $\bit_\mathfrak{t}(G)$ to $\bit(G)$, mapping $X$ to $X/S^1$, and similarly a functor $q$ from $\tor(G_\mathfrak{t})$ to $\tor(G)$.  Observe that $q$ has the structure of a $\bit_\mathfrak{t}(G)$-equivariant $\call$-principal bundle.  Thus a twisted $G$-gerbe on $\mathcal{M}$, which is a principal $\bit_\mathfrak{t}(G)$ bundle $\calp_B$ on $\mathcal{M}$, is a special case of an $S^1$-gerbe $\calp'=\calp_B\times_{\bit_\mathfrak{t}(G)}\tor(G_\mathfrak{t})$ over a $G$-gerbe $\calp=\calp_B\times_{\bit_\mathfrak{t}(G)}\tor(G)$ over $\mathcal{M}$.

We construct the dual of the twisted $G$-gerbe in a very similar manner to the untwisted case.  Namely, the dual data consists of $$\widehat{\calp}=\calp_B\times_{\bit_\mathfrak{t}(G)}\widehat{G}^\mathfrak{t}$$ and a $\call$-principal bundle, i.e., an $S^1$-gerbe $$\widetilde{\calp}=\calp_B\times_{\bit_\mathfrak{t}(G)}\mathfrak{IRep}_\mathfrak{t}(G)$$ over it.

Comparing the $\calp'$-twisted sheaves on $\calp$ to the $\widetilde{\calp}$-twisted sheaves on $\widehat{\calp}$ we have an equivalence $$\Gamma(\calp,\calp'\times_{\call} \text{Vect}_{\mathbb{C}})\simeq\Gamma(\widehat{\calp}, \widetilde{\calp}\times_{\call}\text{Vect}_{\mathbb{C}}).$$ This equivalence, just as the one above, is a consequence of the string of equivalences:
\begin{align*}
\Gamma(\calp,\calp'\times_{\call} \text{Vect}_{\mathbb{C}})&\simeq \Gamma(\calm, \calp_B\times_{\bit_\mathfrak{t}(G)}\mathfrak{Rep}_\mathfrak{t}(G))\\
&\cdots\\
&\simeq\Gamma(\widehat{\calp}, \widetilde{\calp}\times_{\call}\text{Vect}_{\mathbb{C}}).
\end{align*}

Similarly, the remaining questions considered in this paper can be directly translated  into this setting.

\end{document}